\documentclass[fullpage, 11pt]{article}
\usepackage{amsfonts}
\usepackage{amssymb}
\usepackage{amsmath}
\usepackage{latexsym}
\usepackage{amsthm}
\usepackage{epsfig}
\usepackage{graphicx}
\usepackage{subfig}
\usepackage{color}

\newcommand{\bb}{\mathbb}
\newcommand{\R}{\bb R}
\newcommand{\Q}{\bb Q}

\newcommand{\old}[1]{{}}

\newtheorem{theorem}{Theorem}[section]
\newtheorem{corollary}[theorem]{Corollary}
\newtheorem{lemma}[theorem]{Lemma}
\newtheorem{fact}[theorem]{Fact}

\newtheorem{remark}[theorem]{Remark}

\newtheorem{conjecture}[theorem]{Conjecture}

\begin{document}

\title{A Counterexample to a Conjecture of Gomory and Johnson}

\author{
Amitabh Basu \thanks{Supported by a Mellon Fellowship.}   \\
Tepper School of Business, Carnegie Mellon University, Pittsburgh,
PA 15213
\\ abasu1@andrew.cmu.edu
\\ $\;$ \\
Michele Conforti    \\
Department of Mathematics, University of Padova, Italy
\\conforti@math.unipd.it
\\ $\;$ \\
G\'erard Cornu\'ejols
\thanks{Supported by NSF grant CMMI0653419,
ONR grant N00014-03-1-0188 and ANR grant ANR06-BLAN-0375.}   \\
Tepper School of Business, Carnegie Mellon University, Pittsburgh,
PA 15213 \\ and LIF, Facult\'e des Sciences de Luminy, Universit\'e
de Marseille, France
\\ gc0v@andrew.cmu.edu
\\ $\;$ \\
Giacomo Zambelli \\
Department of Mathematics, University of Padova, Italy
\\giacomo@math.unipd.it}

\date{November 2008, Revised November 2009}

\maketitle

\begin{abstract}
In {\em Mathematical Programming} 2003, Gomory and Johnson
conjecture that the facets of the infinite group problem are always
generated by piecewise linear functions. In this paper we give an
example showing that the Gomory-Johnson conjecture is false.
\end{abstract}

\section{Introduction}

Let $f \in ]0,1[$ be given.
Consider the following infinite group problem (Gomory and Johnson \cite{gj}) with a single equality constraint and a nonnegative integer variable $s_r$ associated with each real number $r \in [0,1[$.

\begin{equation} \label{SI}
\begin{array}{rrcl}
          & \sum_{r\in [0,1[} r s_r & = & f                  \\
          &  s_r  & \in & \mathbb{Z}_+ \qquad \forall r\in [0,1[ \\
                   & & s & \mbox{has finite support,}
\end{array}
\end{equation}
\noindent
where additions are performed modulo 1. Vector $s$ has {\em finite support} if $s_r \not= 0$ for a finite number of distinct $r \in [0,1[$.

Gomory and Johnson~\cite{gj} say that a function $\pi : \; [0,1[ \; \rightarrow
\mathbb{R}$ is a {\em valid function} for~\eqref{SI} if $\pi$
is nonnegative, $\pi (0) =0$, and every
solution of (\ref{SI}) satisfies $$\sum_{r\in [0,1[} \pi(r) s_r \geq
1.$$
Note that, if $\pi$ is a valid function, then  $\pi (f) \geq 1$.

For any valid function $\pi$, let $P(\pi)$ be the set of solutions which
satisfy $\sum_{r\in [0,1[} \pi(r) s_r \geq 1$ at equality. An
inequality $\pi$ is a {\em facet} for~\eqref{SI} if and only if $P(\pi^*)
\supset P(\pi)$ implies $\pi^* = \pi$ for every valid function $\pi^*$.

A function is {\em piecewise linear} if
there are finitely many values $0=r_0 < r_1 < \ldots < r_k=1$ such
that the function is of the form $\pi (r) = a_j r + b_j$ in interval
$[r_{j-1},r_j[$, for $j = 1, \ldots k$. The {\em slopes} of a
piecewise linear function are the different values of $a_j$ for $j =
1, \ldots k$.

Gomory and Johnson \cite{sc} gave many examples of facets, and in
all  their examples the facets are piecewise linear. This led them
to formulate the following conjecture, which they describe as
``important and challenging''.

\begin{conjecture}[Facet Conjecture]\label{conj}Every continuous facet for~\eqref{SI} is piecewise linear. \end{conjecture}

We show that the above conjecture is false by exhibiting a facet for
(\ref{SI}) that is  continuous but not piecewise linear.

\subsection{Literature overview}

The definition of valid function we stated, which is the one given
by  Gomory and Johnson in~\cite{gj}, differs from the one adopted
later by the same two authors in~\cite{sc}, where they included in
the definition the further assumption that the function $\pi$ be
continuous. So a ``valid function'' in~\cite{sc} corresponds to a
``continuous valid function'' from ~\cite{gj} and the present paper.
Dey et al.~\cite{DRLM} show that there are facets that are not
continuous.

The definition of facet given in our paper is identical to the one given by Gomory and Johnson in~\cite{sc}, with the caveat that in~\cite{sc} valid functions are required to be continuous. Thus, given a continuous valid function $\pi$, $\pi$ would be a facet according to the definition in~\cite{sc} if and only if $P(\pi^*)\supset P(\pi)$ implies $\pi^* = \pi$ for every {\em continuous} valid function $\pi^*$. For clarity, we refer to a continuous valid function satisfying the latter property as a {\em facet in the sense of Gomory-Johnson}. By definition, if a continuous valid function is a facet, it is also a facet in the sense of Gomory-Johnson. Thus, the conjecture they state in~\cite{sc} is actually that {\em every facet in the sense of Gomory-Johnson is piecewise linear}. The above discussion shows that a counterexample to Conjecture~\ref{conj} also disproves the conjecture in~\cite{sc}.
\smallskip

In earlier work, Gomory and Johnson~\cite{gj} emphasized extreme
functions rather than facets. A valid function $\pi$ is {\em
extreme} if it cannot be expressed as a convex combination of two
distinct valid functions. It follows from the definition that facets
are extreme. Therefore our counterexample also provides an extreme
function that is continuous but not piecewise linear.
Gomory and Johnson~\cite{sc} write in a footnote that the definition
of facet ``is different from, although eventually equivalent to,''
the definition of extreme function. The statement that extreme
functions are facets appears to be quite nontrivial to prove, and to
the best of our knowledge there is no proof in the literature. We
therefore cautiously treat extreme functions and facets as distinct
concepts, and leave their equivalence as an open question.

We obtain the counterexample to Conjecture~\ref{conj} by exhibiting
a sequence of piecewise linear functions and then considering the
pointwise limit of this sequence of functions. These functions were 
discovered by Kianfar and Fathi~\cite{kf} who show that they are
facets for the infinite group problem. They call them the $n$-step MIR (Mixed-Integer Rounding) functions in their paper. Our treatment and analysis in this paper is different from theirs. The emphasis in~\cite{kf} was on deriving valid inequalities for MILPs which are generalizations of the standard MIR inequalities. In this paper, we
use this class of functions primarily to construct a
counterexample to Conjecture~\ref{conj}.

\section{Preliminaries}

A valid function $\pi : \; [0,1[ \; \rightarrow \mathbb{R}$ is {\em
minimal} if there is no valid function $\pi'$ such that 1) $\pi'(a)
\leq \pi (a)$ for all $a \in [0,1[$, and 2) the inequality is strict
for at least one $a$. If $\pi$ is a minimal valid function, then $\pi(r)\leq 1$ for every $r\in [0,1[$, as follows immediately from the fact that $\pi$ is nonnegative.

When convenient, we will extend the domain of definition of the
function $\pi$ to the whole real line $\mathbb{R}$ by making the
function periodic: $\pi (x) = \pi (x + k)$ for any $x \in [ 0,1 [$
and $k \in \mathbb{Z}$.

A function $\pi : \; \mathbb{R} \; \rightarrow \mathbb{R}$ is  {\em
subadditive} if for every $a,b \in \mathbb{R}$
\[
\pi(a + b) \leq \pi(a) + \pi(b).
\]

Given $f \in ] 0, 1 [$, a function $\pi : \; [0,1[ \; \rightarrow
\mathbb{R}$ is {\em symmetric} if for every $a \in [ 0,1 [$
\[
\pi(a)+\pi(f-a)=1.
\]

Gomory and Johnson prove the following result in \cite{gj}.

\begin{theorem}[Minimality Theorem]\label{thm:gj}
Let $\pi : \; [0,1[ \; \rightarrow \mathbb{R}$ be such that
$\pi(0)=0$ and $\pi(f)=1$. A necessary and sufficient condition for
$\pi$ to be valid and minimal is that $\pi$ is subadditive and
symmetric.
\end{theorem}

Any facet for (\ref{SI}) is minimal. Therefore if $\pi$ is a facet
for (\ref{SI}), then $\pi$ is subadditive and symmetric.
The following two facts are well-known and will be useful in our
arguments.

\begin{fact}\label{fact:subadd+linear}
If $\pi_1$ and $\pi_2$ are subadditive, then $\pi_1 + \pi_2$ is
subadditive.
\end{fact}


\begin{fact}\label{fact:scaling_subadd}
Let $\pi$ be a subadditive function and define $\pi'(x) =
\alpha\pi(\beta x)$ for some constants $\alpha > 0$ and $\beta$.
Then $\pi'$ is subadditive.
\end{fact}

\begin{proof}
$\pi'(a+b) = \alpha\pi(\beta(a+b)) \leq \alpha(\pi(\beta
a)+\pi(\beta b)) = \pi'(a) + \pi'(b)$.
\end{proof}

A function $\pi'$ defined as $\pi'(x)= \alpha\pi(\beta x)$ will be
referred to as a \emph{scaling} of $\pi$.
\medskip

Let $E(\pi)$ denote the set of all possible of inequalities
$\pi(u_1) + \pi(u_2) \geq \pi(u_1 + u_2)$ that are satisfied as
\emph{equalities} by $\pi$. Here $u_1$ and $u_2$ are any real
numbers. The following theorem is proved in Gomory and
Johnson~\cite{sc} and is used in this paper to prove that certain
inequalities are facets.

\begin{theorem}[Facet Theorem]\label{thm:facet_thm}
Let $\pi$ be a minimal valid function. If there is no minimal valid
function that satisfies the equalities in $E(\pi)$ other than $\pi$
itself, then $\pi$ is a facet.
\end{theorem}

We remark that, even though Theorem~\ref{thm:facet_thm} is proved in~\cite{sc} under the assumptions that valid functions are continuous, the continuity assumption is not needed in the proof, thus the statement remains true even in the setting of the present paper.

In the paper we need the following lemma, which is a variant of the Interval Lemma stated in
Gomory and Johnson~\cite{sc}. They prove the lemma under the assumption that the function in the statement is continuous, whereas we only require the function to be bounded one every interval. Other variants of the Interval Lemma that do not require the function to be continuous have been given by Dey et al.~\cite{DRLM}. The proof we give is in the same spirit of the solution of Cauchy's Equation~(see for example Chapter 2 of Acz\'el~\cite{acz}).

\begin{lemma}[Interval Lemma] \label{interval_lemma} Let $\pi : \; \R \; \rightarrow
\mathbb{R}$ be a function bounded on every bounded interval. Given real numbers $u_1 < u_2$ and $v_1 < v_2$, let $U = [u_1, u_2]$, $V =
[v_1, v_2]$, and $U + V = [u_1 + v_1, u_2 + v_2]$.\\ If $\pi(u)+\pi(v) = \pi(u+v)$ for every
$u \in U$ and $v \in V$, then there exists $c\in \R$ such that $\pi(u)=\pi(u_1)+c(u-u_1)$ for every $u\in U$, $\pi(v)=\pi(v_1)+c(v-v_1)$ for every $v\in V$, $\pi(w)=\pi(u_1+v_1)+c(w-u_1-v_1)$ for every $w\in U+V$.
\end{lemma}

\begin{proof} We first show the following.
\medskip

\noindent{\em Claim 1. Let $u \in U$,  and let $\varepsilon>0$ such that $v_1+\varepsilon\in V$. For every nonnegative integer $p$ such that $u+p\varepsilon\in U$, we have $\pi(u + p\varepsilon) - \pi(u) = p(\pi(v_1 + \varepsilon) -
\pi(v_1))$.}
\medskip

For $h=1\ldots, p$, by hypothesis $\pi(u + h\varepsilon) + \pi(v_1) =
\pi(u + h\varepsilon + v_1) = \pi(u+(h-1)\varepsilon) + \pi( v_1+
\varepsilon)$. Thus $\pi(u + h\varepsilon) - \pi(u+(h-1)\varepsilon) = \pi(v_1 + \varepsilon) -
\pi(v_1)$, for $h=1,\ldots,p.$
By summing the above $p$ equations, we obtain $\pi(u + p\varepsilon) - \pi(u) = p(\pi(v_1 + \varepsilon) -
\pi(v_1))$. This concludes the proof of Claim~1.\medskip

Let $\bar u, \bar u'\in U$ such that $\bar u-\bar u'\in \Q$ and
$\bar u>\bar u'$.  Define $c:=\frac{\pi(\bar u)-\pi(\bar u')}{\bar
u-\bar u'}$.
\bigskip

\noindent{\em Claim 2. For every $u,u'\in U$ such that $u-u'\in \Q$,
we have $\pi(u)-\pi(u')=c(u-u')$. }\medskip

 We only need to show that, given $u,u'\in U$
such that $u-u'\in \Q$, we have $\pi(u)-\pi(u')=c(u-u')$. We may
assume $u>u'$. Choose a positive rational $\varepsilon$ such that
$\bar u-\bar u'=\bar p \varepsilon$ for some integer $\bar p$, $ u-
u'= p \varepsilon$ for some integer $ p$, and $v_1+\varepsilon \in
V$. By Claim~1,
$$\pi(\bar u)-\pi(\bar u')=\bar p (\pi(v_1 + \varepsilon) -
\pi(v_1)) \quad \mbox{
and
}\quad \pi( u)-\pi( u')= p (\pi(v_1 + \varepsilon) -
\pi(v_1)).$$
Dividing the last equality by $u-u'$ and the second to last by $\bar u-\bar u'$, we get
$$\frac{\pi(v_1 + \varepsilon) -
\pi(v_1)}{\varepsilon}=\frac{\pi(\bar u)-\pi(\bar u')}{\bar u-\bar u'}=\frac{\pi( u)-\pi( u')}{ u- u'}=c.$$
Thus $\pi(u)-\pi(u')=c(u-u')$. This concludes the proof of Claim~2.
\bigskip

\noindent{\em Claim 3. For every $u\in U$, $\pi(u) = \pi(u_1) + c
(u-u_1)$.}
\medskip

Let $\delta(x) = \pi(x) - c  x$. We show that  $\delta(u) =
\delta(u_1)$ for all $u \in U$ and this proves the claim.  Since
$\pi$ is bounded on every bounded interval, $\delta$ is bounded over
$U, V$ and $U+V$. Let $M$ be a number such that $|\delta(x)| \leq M$
for all $x\in U\cup V\cup (U+V)$.

Suppose by contradiction that, for some $u^*\in U$, $\delta(u^*)
\neq \delta(u_1)$. Let $N$ be a positive integer such that
$|N(\delta(u^*)-\delta(u_1))| > 2M$.\\ By Claim~2, $\delta(u^*) =
\delta(u)$ for every $u\in U$ such that $u^*-u$ is rational. Thus
there exists  $\bar u$ such  that $\delta(\bar u)=\delta(u^*)$, $u_1
+ N(\bar u - u_1) \in U$ and $v_1 + \bar u - u_1 \in V$. Let $\bar u
- u_1 = \varepsilon$. By Claim~1,
\[
\delta(u_1 + N\varepsilon) - \delta(u_1)
=N(\delta(v_1+\varepsilon)-\delta(v_1)=N(\delta(u_1 + \varepsilon) -
\delta(u_1) ) =N(\delta(\bar u)-\delta(u_1))
\]
Thus $|\delta(u_1 + N\varepsilon) - \delta(u_1)| =  |N(\delta(\bar
u)-\delta(u_1))|=|N(\delta(u^*)-\delta(u_1))| > 2M$, which implies
$|\delta(u_1 + N\varepsilon)|+ |\delta(u_1)|>2M$, a contradiction.
This concludes the proof of Claim~3.
\bigskip

By symmetry between $U$ and $V$, Claim~3 implies  that there exists
some constant $c'$ such that, for every $v\in V$,
$\pi(v)=\pi(v_1)+c'(v-v_1)$. We show $c'=c$. Indeed, given
$\varepsilon>0$ such that $u_1+\varepsilon\in U$ and
$v_1+\varepsilon\in V$, $c\varepsilon=\pi(u_1 + \varepsilon) -
\pi(u_1) = \pi(v_1 + \varepsilon) -\pi(v_1)=c'\varepsilon$, where
the second equality follows from Claim~1.\\ Therefore, for every
$v\in V$, $\pi(v) = \pi(v_1) + c  \pi(v-v_1)$. Finally, since
$\pi(u)+\pi(v)=\pi(u+v)$ for every $u\in U$ and $v\in V$, for every
$w \in U+V$, $\pi(w) = \pi(u_1+v_1) + c  (w-u_1-v_1)$. \end{proof}

The following theorem, due to Gomory and Johnson~\cite{sc}, gives a
class of facets.

\begin{theorem}\label{thm:2-slopes} Let $\pi : \; [0,1[ \; \rightarrow \mathbb{R}$
be a minimal valid function that is piecewise linear. If $\pi$ has
only two distinct slopes, then $\pi$ is a facet.
\end{theorem}

Again, we note that Theorem~\ref{thm:2-slopes} is proved
in~\cite{sc} under the assumptions that valid functions are
continuous. However, in~\cite{sc} the continuity assumption is used
in the proof only when applying the Interval Lemma. Since our
version of the Interval Lemma (Lemma~\ref{interval_lemma}) applies
to any bounded function, and since minimal valid functions are
bounded, Theorem~\ref{thm:2-slopes} is valid also in the setting of
the present paper. Gomory and Johnson give in~\cite{gj2} a very
similar statement to the one of Theorem~\ref{thm:2-slopes}, namely
that piecewise linear minimal valid functions with only two distinct
slopes are extreme.

\section{The construction}
\label{sec:constr}

We first define a sequence of valid functions $\psi_i : \; [0,1[ \;
\rightarrow \mathbb{R}$ that are piecewise linear, and then consider
the limit $\psi$ of this sequence. We will then show that $\psi$ is
a facet but not piecewise linear.

Let $0 < \alpha < 1$. $\psi_0$ is the triangular function given by

\[
\psi_0(x) = \left\{
\begin{array}{ll}
\frac{1}{\alpha}x & 0 \leq x  \leq \alpha \\
\frac{1-x}{1-\alpha} & \alpha \leq x  < 1.
\end{array}\right.
\]

Notice that the corresponding inequality
$\sum_{r\in[0,1[}\psi_0(r)s_r\geq 1$ defines the Gomory
mixed-integer inequality if one views (\ref{SI}) as a relaxation of
the simplex tableau of an integer program.

\begin{figure}[htbp]
\begin{center}
\scalebox{.4}{\epsfig{file=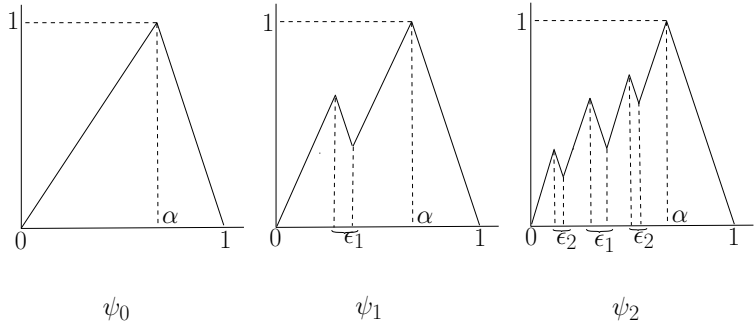}} \caption{First two
steps in the construction of the limit function}
\label{fig:limit_function}
\end{center}
\end{figure}

We first fix a nonincreasing sequence of positive real numbers
$\epsilon_i, \textrm{for }i = 1,2,3, \ldots$, such that $\epsilon_1
\leq 1-\alpha$ and

\begin{equation} \label{series}
\sum_{i=1}^{+\infty} 2^{i-1} \epsilon_i < \alpha.
\end{equation}

For example, $ \epsilon_i = \alpha(\frac{1}{4})^i$ is such a
sequence when $0 < \alpha \leq \frac{4}{5}$. The upper bound of
$\frac{4}{5}$ for $\alpha$ is implied by the fact that $\epsilon_1
\leq 1 - \alpha$.\bigskip

We construct $\psi_{i+1}$ from $\psi_i$ by modifying each segment
with positive slope in the graph of $\psi_i$ as follows.

For every  maximal (with respect to set inclusion) interval $[a,b]
\subseteq [0,\alpha]$ where $\psi_i$ has constant positive slope we
replace the line segment from $(a,\psi_i(a))$ to $(b,\psi_i(b))$
with the following three segments.
\begin{itemize}
\item  The segment connecting $(a,\psi_i(a))$
and
$(\frac{(a+b)-\epsilon_{i+1}}{2},\psi_i(\frac{a+b}{2})+\frac{\epsilon_{i+1}}{2(1-\alpha)})$,
\item The segment connecting  $(\frac{(a+b)-\epsilon_{i+1}}{2},
\psi_i(\frac{a+b}{2})+\frac{\epsilon_{i+1}}{2(1-\alpha)})$ and
$(\frac{(a+b)+\epsilon_{i+1}}{2},
\psi_i(\frac{a+b}{2})-\frac{\epsilon_{i+1}}{2(1-\alpha)})$,
\item The segment connecting $(\frac{(a+b)+\epsilon_{i+1}}{2},
\psi_i(\frac{a+b}{2})-\frac{\epsilon_{i+1}}{2(1-\alpha)})$ and $(b,
\psi_i(b))$.
\end{itemize}
Figure~\ref{fig:limit_function} shows the transformation of $\psi_0$
to $\psi_1$ and $\psi_1$ to $\psi_2$.

The function $\psi$ which we show to be a facet but not piecewise
linear is defined as the limit of this sequence of functions, namely

\begin{equation}
\psi(x) = \lim_{i \rightarrow \infty} \psi_i(x)
\end{equation}\label{PSI}

This limit is well defined when (\ref{series}) holds, as shown in
Section~\ref{SEC-limit}.

In the next section we show that each function $\psi_i$ is well
defined and is a facet. In Section~\ref{SEC-limit} we analyze the
limit function $\psi$, showing that it is well defined, is a facet,
but is not piecewise linear.

As discussed in the Introduction, the sequence $\psi_i$ defines a
class of facets which was also discovered independently by Kianfar
and Fathi in~\cite{kf} where they are referred to as $n$-step MIR
functions. Their constructions are a little more general than
the class of facets defined by the $\psi_i$'s. Our
analysis in the next section is different from their treatment of
these functions.

\section{Analysis of the function $\psi_i$}

\begin{fact}\label{fact:measures} For $i\geq 0$,
$\psi_i$ is a continuous  function which is  piecewise linear
 with $2^i$ pieces with positive slope and  $2^i$ pieces
with negative slope. Furthermore:
\begin{enumerate}
\item There is one negative slope interval of length
$1-\alpha$ and there are $2^{k-1}$ negative slope intervals of
length $\epsilon_k$ for $k=1,\ldots,i$;
\item The negative slope pieces have  slope
$-\frac{1}{1-\alpha}$;
\item Each positive slope interval has length $\frac{\gamma_i}{2^i}$, where
$\gamma_i={\alpha-\sum_{k=1}^i 2^{k-1}\epsilon_k}$;
\item The positive slope pieces have slope
$\frac{1-\gamma_i}{(1-\alpha)\gamma_i}$;
\item The function $\psi_i$ is well-defined.
\end{enumerate}
\end{fact}

\begin{proof}  The fact that
$\psi_i$ is a continuous  function which is  piecewise linear
 with $2^i$ pieces with positive slope and  $2^i$ pieces
with negative slope, and facts 1. and 2. are immediate by
construction.\\ Therefore the sum of the lengths of the negative
slope intervals is $1-\alpha+\sum_{k=1}^i 2^{k-1}\epsilon_k$. Since
$\psi_i$ contains $2^i$ positive slope intervals with the same
length, this
proves 3.\\
The total decrease of $\psi_i$ in $[0,1]$ is
$\frac{-1}{1-\alpha}(1-\gamma_i)$. Since $\psi_i$ is continuous,
piecewise linear, all positive slope intervals have the same slope
and $\psi_i(0)=\psi_i(1)=0$, then a positive slope interval has
slope $\frac{1-\gamma_i}{(1-\alpha)\gamma_i}$ and this proves 4.\\
Finally, by~(\ref{series}), $\gamma_i>0$ for every $i\geq 0$, thus
$\psi_i$ is a well-defined function.

\end{proof}

We now demonstrate that each function $\psi_i$ is subadditive.  Note
that the function $\psi_i$ depends only upon the choice of
parameters $\alpha, \epsilon_1, \epsilon_2, \ldots, \epsilon_i$. It
will sometimes be convenient to denote the function $\psi_i$ by
$\psi_i^{\alpha,\epsilon_1, \epsilon_2, \ldots, \epsilon_i}$ in this
section.

\begin{figure}[htbp]
\begin{center}
\scalebox{.4}{\epsfig{file=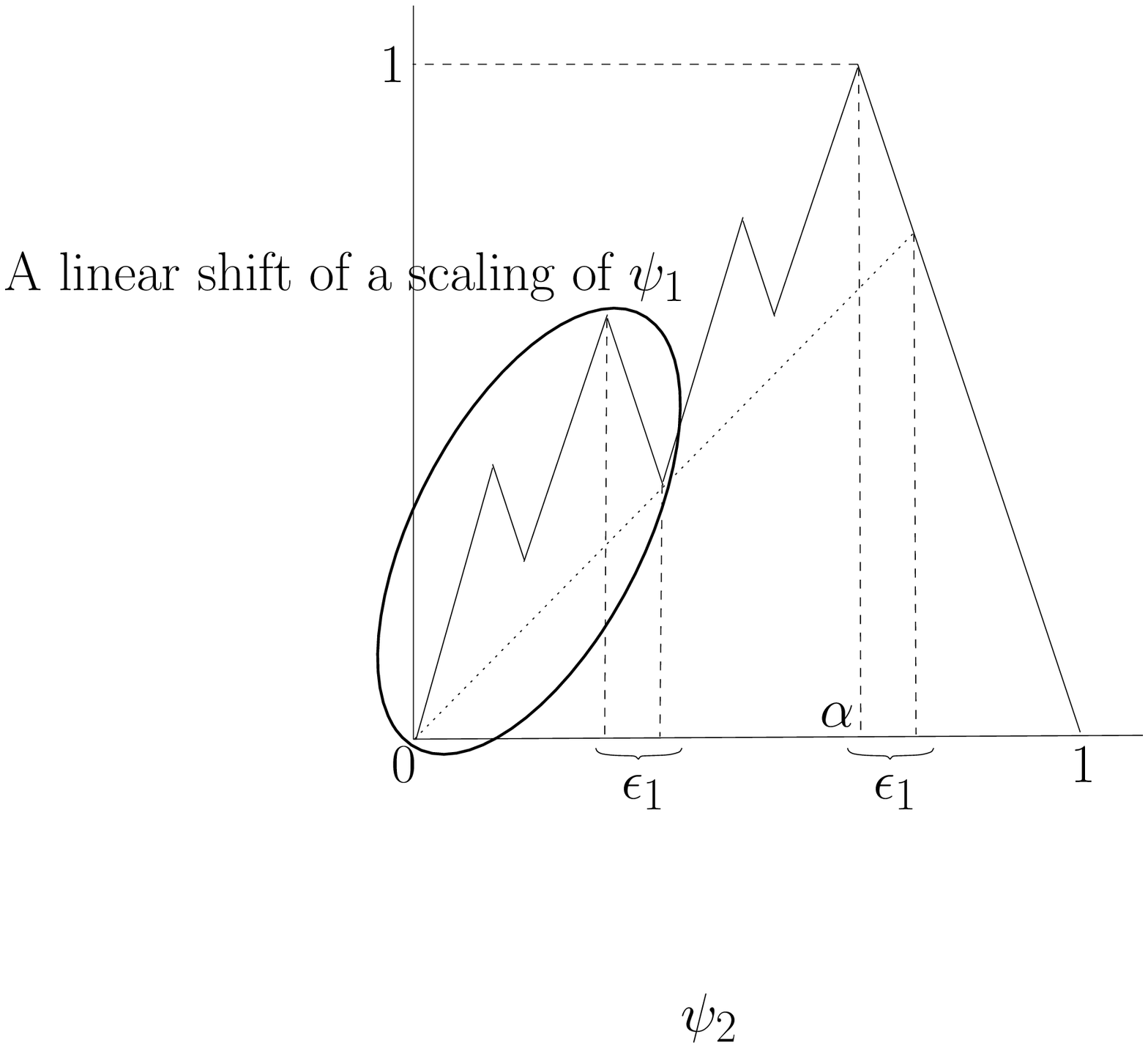}}
\caption{Illustrating the proof of subadditivity of $\psi_2$}
\label{fig:subadditivity_proof}
\end{center}
\end{figure}

The key observation is the following lemma.
Figure~\ref{fig:subadditivity_proof} illustrates this for the
function $\psi_2^{\alpha,\epsilon_1,\epsilon_2}$.

\begin{lemma}\label{lem:recursive lemma}
For $x \in [0, \alpha + \epsilon_1]$ and $i\geq 1$, $\psi_i^{\alpha,
\epsilon_1, \epsilon_2, \ldots, \epsilon_i}(x) = \lambda x + \mu
\psi_{i-1}^{\frac{\alpha-\epsilon_1}{\alpha+\epsilon_1},
\frac{2\epsilon_2}{\alpha+\epsilon_1},
\frac{2\epsilon_3}{\alpha+\epsilon_1},\ldots,
\frac{2\epsilon_i}{\alpha+\epsilon_1}}(\frac{2x}{\alpha+\epsilon_1})$,
where
\[
\lambda = \frac{1-\alpha -
\epsilon_1}{(\alpha+\epsilon_1)(1-\alpha)} \textrm{and } \mu =
\frac{\epsilon_1}{(\alpha + \epsilon_1)(1 - \alpha)} .
\]
\end{lemma}

\begin{proof}
Notice that $\lambda$ is the slope of the line passing through the
points $(0,0)$ and $(\frac
{\alpha+\epsilon_1}2,\psi^{\alpha,\epsilon_1}_1(\frac
{\alpha+\epsilon_1}2))=(\frac {\alpha+\epsilon_1}2,\frac
{1-\epsilon_1}2)$.

For $x\in[0,\alpha+\epsilon_1]$, let $\phi_i(x) = \psi_i^{\alpha,
\epsilon_1, \epsilon_2, \ldots, \epsilon_i}(x) - \lambda x$. Notice
that the graph of $\phi_1$ in the interval $[0,\alpha+\epsilon_1]$
is comprised of two identical triangles, one with basis
$[0,\frac{\alpha+\epsilon_1}2]$ and apex
$(\frac{\alpha-\epsilon_1}2,\mu)$ and the other with basis
$[\frac{\alpha+\epsilon_1}2,\alpha+\epsilon_1]$ and apex
$(\alpha,\mu)$, where
$\mu=\psi_1(\frac{\alpha-\epsilon_1}2)-\lambda(\frac{\alpha-\epsilon_1}2)$.
Therefore, for $x\in [0,1[$,
$$\mu^{-1}\phi_1\left(\frac{(\alpha+\epsilon_1)x}2\right)=\psi_0^{\frac{\alpha-\epsilon_1}{\alpha+\epsilon_1}}(x)$$
thus
$\phi_1(x)=\mu\psi_0^{\frac{\alpha-\epsilon_1}{\alpha+\epsilon_1}}(\frac{2x}{\alpha+\epsilon_1})$
because $\phi_1(x)=\phi_1(x+\frac{\alpha+\epsilon_1}2)$ for every
$x\in[0,\frac{\alpha+\epsilon_1}2[$.\bigskip

Assume by induction that
$\phi_i(x)=\phi_i(x+\frac{\alpha+\epsilon_1}2)$ for every
$x\in[0,\frac{\alpha+\epsilon_1}2[$, and that, for $x \in [0,1[$,
$\mu^{-1}\phi_i\left(\frac{(\alpha+\epsilon_1)x}2\right)=
\psi_{i-1}^{\frac{\alpha-\epsilon_1}{\alpha+\epsilon_1},\frac{2\epsilon_2}{\alpha+\epsilon_1},\ldots,\frac{2\epsilon_{i}}{\alpha+\epsilon_1}}(x)$.

Notice that, by Fact~\ref{fact:measures}, the slope of $\psi_i$ in
the intervals of positive slope is always greater than $\lambda$,
hence $\phi_i$ has positive slope exactly in the same intervals
where $\psi_i$ has positive slope.

Therefore, by construction of
$\psi^{\alpha,\epsilon_1,\ldots,\epsilon_{i+1}}_{i+1}$, the function
$\phi_{i+1}$ is obtained from $\phi_i$ by replacing each maximal
positive slope segment $[(a,\phi_i(a)),(b,\phi_i(b))]$ with: \\
- the segment connecting
$(a,\phi_i(a))$ and $(\frac{(a+b)-\epsilon_{i+1}}{2},\phi_i(\frac{a+b}{2})+\frac{\epsilon_{i+1}}{2(1-\alpha)})$,\\
- the segment connecting
$(\frac{(a+b)-\epsilon_{i+1}}{2},\phi_i(\frac{a+b}{2})+\frac{\epsilon_{i+1}}{2(1-\alpha)})$
and $(\frac{(a+b)+\epsilon_{i+1}}{2},
\phi_i(\frac{a+b}{2})-\frac{\epsilon_{i+1}}{2(1-\alpha)})$,\\
- the segment connecting  $(\frac{(a+b)+\epsilon_{i+1}}{2},
\phi_i(\frac{a+b}{2})-\frac{\epsilon_{i+1}}{2(1-\alpha)})$ and $(b,
\phi_i(b))$.

Thus, by induction,
$\phi_{i+1}(x)=\phi_{i+1}(x+\frac{\alpha+\epsilon_1}2)$ for every
$x\in[0,\frac{\alpha+\epsilon_1}2[$ and, for $x\in[0,1[$, we have
that $\mu^{-1}\phi_{i+1}\left(\frac{(\alpha+\epsilon_1)x}2\right)=
\psi_{i}^{\frac{\alpha-\epsilon_1}
{\alpha+\epsilon_1},\frac{2\epsilon_2}{\alpha+\epsilon_1},\ldots,\frac{2\epsilon_{i}}{\alpha+\epsilon_1},\frac{2\epsilon_{i+1}}{\alpha+\epsilon_1}}(x)$.
Therefore, for $x\in[0,\alpha+\epsilon_1[$, we have $\phi_{i+1}(x)=
\mu\psi_{i}^{\frac{\alpha-\epsilon_1}
{\alpha+\epsilon_1},\frac{2\epsilon_2}{\alpha+\epsilon_1},\ldots,\frac{2\epsilon_{i+1}}{\alpha+\epsilon_1}}(\frac{2x}{\alpha+\epsilon_1})$.
\end{proof}

\begin{remark}\label{rmk:coefficients}
Given any $0 < \alpha <1$ and any non-increasing sequence of
positive real numbers $\epsilon_i$ satisfying (\ref{series}) and
$\epsilon_1 \leq 1-\alpha$, let
$\alpha'=\frac{\alpha-\epsilon_1}{\alpha+\epsilon_1}$,
$\epsilon'_i=\frac{2\epsilon_{i+1}}{\alpha+\epsilon_1}$, $i\geq 1$.
Then $\epsilon'_1\leq 1-\alpha'$, $\{\epsilon_i'\}$ is a
non-increasing sequence, and $\sum_{i=1}^{+\infty}
2^{i-1}\epsilon'_i<\alpha'$.
\end{remark}

We next prove that each $\psi_i$ is a non-negative function.
\begin{fact}\label{fact:non-neg}
$\psi_i^{\alpha, \epsilon_1, \ldots \epsilon_i}(x) \geq 0$ for all
$x$, and for all parameters such that $\epsilon_1 \leq 1-\alpha$ and
$\epsilon_i$ is a non-increasing sequence.
\end{fact}

\begin{proof}
The proof is by induction on $i$. $\psi_0$ is non-negative by
definition.

Consider $\psi_{i+1}$. Clearly $\psi_{i+1}(x)\geq 0$ for $x \in
[\alpha + \epsilon_1, 1[$, since $\psi_{i+1}(x) = \psi_0(x)$ in this
interval. Note that in Lemma~\ref{lem:recursive lemma}, $\lambda
\geq 0$ because $1-\alpha \geq \epsilon_1$. So, when $x \in
[0,\alpha+\epsilon_1]$ Lemma~\ref{lem:recursive lemma} implies that
$\psi_{i+1}$ is non-negative, because $\psi_i$ is non-negative. Note
that the parameters for $\psi_i$ also satisfy the hypothesis by
Remark~\ref{rmk:coefficients}, so we can use the induction
hypothesis.
\end{proof}

\begin{lemma}
Given any $0 < \alpha <1$ and any nonincreasing sequence of positive
real numbers $\epsilon_i$ satisfying (\ref{series}) and $\epsilon_1
\leq 1-\alpha$, the function $\psi_i^{\alpha,\epsilon_1, \epsilon_2,
\ldots, \epsilon_i}$ is subadditive for all $i$.
\end{lemma}

\begin{proof}
The proof is by induction. $\psi_0^{\alpha}$ is subadditive, since
it is a valid and minimal Gomory function. By the induction
hypothesis, $\psi_k^{\alpha, \epsilon_1, \ldots, \epsilon_k}$ is
subadditive and we wish to show this implies that
$\psi_{k+1}^{\alpha, \epsilon_1,\ldots,\epsilon_{k+1}}$ is
subadditive.

By Remark \ref{rmk:coefficients} and induction, the function
$\psi_{k}^{\frac{\alpha-\epsilon_1}{\alpha+\epsilon_1},
\frac{2\epsilon_2}{\alpha+\epsilon_1},
\frac{2\epsilon_3}{\alpha+\epsilon_1},\ldots,
\frac{2\epsilon_{k+1}}{\alpha+\epsilon_1}}$ is subadditive.\\

 We
define the function $\psi'_k(x) = \mu
\psi_{k}^{\frac{\alpha-\epsilon_1}{\alpha+\epsilon_1},
\frac{2\epsilon_2}{\alpha+\epsilon_1},
\frac{2\epsilon_3}{\alpha+\epsilon_1},\ldots,
\frac{2\epsilon_{k+1}}{\alpha+\epsilon_1}}(\frac{2x}{\alpha+\epsilon_1})$
where $\mu$ is defined in the statement of Lemma~\ref{lem:recursive
lemma}. Note that $\psi'_k$ has a period of
$\frac{\alpha+\epsilon_1}{2}$ and $\psi'_k$ is subadditive by
Fact~\ref{fact:scaling_subadd}.

In the remaining part of the proof, we will not need the extended
notation $\psi_k^{\alpha, \epsilon_1, \epsilon_2, \ldots,
\epsilon_k}$ and we shall refer to this function as simply $\psi_k$.
We now prove the subadditivity of $\psi_{k+1}$ assuming the
subadditivity of $\psi_k$, i.e. $\psi_{k+1}(a+b) \leq \psi_{k+1}(a)
+ \psi_{k+1}(b)$. Assume without loss of generality that $a\leq b$.
We then have the following two cases.

\medskip
\emph{Case 1 : $b$ is in the range $[0,\alpha+\epsilon_1]$}.\\
If $a+b \in [0,\alpha+\epsilon_1]$, $\psi_{k+1}(a+b) = \psi'_k(a+b)
+ \lambda(a+b)$ by Lemma~\ref{lem:recursive lemma}. Now
Fact~\ref{fact:subadd+linear} shows that $\psi_{k+1}$ is
subadditive.\\
If $a+b$ is in the range $[\alpha+\epsilon_1, 1]$, then
$\psi_{k+1}(a+b) \leq \lambda(a+b)$. On the other hand,
$\psi_{k+1}(a) \geq \lambda a$ and $\psi_{k+1}(b) \geq \lambda b$,
hence $\psi_{k+1}(a+b)\leq \psi_{k+1}(a)+\psi_{k+1}(b)$.\\ If $a+b$
is greater than $1$, then $(a+b)\textrm{mod }1<\alpha+\epsilon_1$.
Let $x = \alpha+\epsilon_1 - b$ and $y = 1 - \alpha-\epsilon_1$. So
$(a + b)\textrm{mod }1 = a - x - y$. Then

\begin{equation}\label{eq3}
\begin{array}{rclr}
\psi_{k+1}(a) + \psi_{k+1}(b) & = & \psi'_k(a) + \psi'_k(b) + \lambda a + \lambda b &\\
 & \geq & \psi'_k(a + b) +  \lambda a +  \lambda b &((a+b)\textrm{mod }1<\alpha+\epsilon_1)\\
 & = & \psi'_k(a - x) +  \lambda a +  \lambda b& (a-x=(a+b)-(\alpha+\epsilon_1))
\end{array}
\end{equation}
\noindent because $\psi'_k$ has period $\alpha+\epsilon_1$. Also,

\begin{equation}\label{eq4}
\begin{array}{rclr}
\psi_{k+1}(a + b ) & = & \psi_{k+1}(a -x- y) & \\
 & \leq & \psi_{k+1}(a-x) + \frac{y}{1-\alpha} & (\textrm{All negative slopes in }\psi_{k+1} \textrm{ are }-\frac{1}{1-\alpha})\\
 & = & \psi_{k+1}(a-x) +  \lambda (b + x) & (\textrm{by definition of } \lambda, x, y)\\
 & = & \psi'_k(a-x) +  \lambda (a-x) +  \lambda (b+x) & (\textrm{by Lemma \ref{lem:recursive lemma} because } 0 \leq a-x \leq \alpha + \epsilon_1)\\
 & = & \psi'_k(a-x) +  \lambda a +  \lambda b&
\end{array}
\end{equation}

\old{
\begin{equation}\label{eq4}
\begin{array}{rclr}
\psi_{k+1}(a + b ) & = & \psi_{k+1}((a -x) +(1- y)) & \\
 & \leq & \psi_{k+1}(a-x) + \psi_{k+1}(1-y) & ((a + b)\textrm{mod }1,\,a-x,\,1-y \textrm{ are in } [0,\alpha+\epsilon_1])\\
 & = & \psi_{k+1}(a-x) + \frac{1-\alpha-\epsilon_1}{1-\alpha} & (\textrm{All negative slopes in }\psi_{k+1} \textrm{ are }-\frac{1}{1-\alpha})\\
  & = & \psi_{k+1}(a-x) +  \lambda (b + x) & (\textrm{by definition of } x, y)\\
 & = & \psi'_k(a-x) +  \lambda (a-x) +  \lambda (b+x) & (\textrm{by Lemma \ref{lem:recursive lemma} because } 0 \leq a-x \leq \alpha + \epsilon_1)\\
 & = & \psi'_k(a-x) +  \lambda a +  \lambda b&
\end{array}
\end{equation}
}From (\ref{eq3}) and (\ref{eq4}) we get that $\psi_{k+1}(a) +
\psi_{k+1}(b) \geq \psi_{k+1}(a+b)$.

\old{
\[\begin{array}{rcl}
\psi_{k+1}(a+b) & = & \psi_{k+1}((a+b) \textrm{mod }1)) \\
 & = & \psi'_k((a+b)\textrm{mod }1) + l((a+b)\textrm{mod }1) \\
 & \leq & \psi'_k(a+b) + l(a+b) \\
 & \leq & \psi'_k(a) + \psi'(b) + l(a) + l(b) \\
 & = & \psi_{k+1}(a) + \psi_{k+1}(b)
\end{array}
\]
\noindent where the second inequality follows from the subadditivity
of $\psi'_k$. }

\medskip
\emph{Case 2 : $b$ is in the range $[\alpha+\epsilon_1, 1]$}

If $a+b$ is also in the range $[\alpha+\epsilon_1, 1]$, then
$\psi_{k+1}(a+b) \leq \psi_{k+1}(b)$. Therefore, $\psi_{k+1}(a+b)
\leq \psi_{k+1}(a) + \psi_{k+1}(b)$.

Now we consider the case where $a+b > 1$. Since every line segment
with negative slope in the graph of $\psi_{k+1}$ has slope
$-\frac{1}{1-\alpha}$, then in the range $[0,a]$ the line of slope
$-\frac{1}{1-\alpha}$ passing through $(a,\psi_{k+1}(a))$ lies above
the graph of $\psi_{k+1}$. Formally, for every $x\in [0,a]$,
\begin{equation}\label{eq:LAST-CASE}-\frac{1}{1-\alpha}x+\psi_{k+1}(a)+\frac{a}{1-\alpha}\geq\psi_{k+1}(x).\end{equation}

Now
$$\psi_{k+1}(a)+\psi_{k+1}(b)=\psi_{k+1}(a)+\frac{1-b}{1-\alpha}\geq-\frac{a+b-1}{1-\alpha}+\psi_{k+1}(a)+\frac{a}{1-\alpha}\geq\psi_{k+1}(a+b-1)$$
where the first equality is because $\psi_{k+1}(b) =
\frac{1}{1-\alpha}(1-b)$, and the last inequality follows
by~(\ref{eq:LAST-CASE}). Therefore, we get $\psi_{k+1}(a) +
\psi_{k+1}(b) \geq \psi_{k+1}(a+b-1) = \psi_{k+1}(a+b)$.
\end{proof}

\begin{fact}
$\psi_i(x)$ is a symmetric function.
\end{fact}

\begin{proof} It is straightforward to show that $\psi_0$ is
symmetric. Notice that, by construction, the function
$\psi_{i+1}-\psi_i $ satisfies
$$(\psi_{i+1}-\psi_i)(x)+(\psi_{i+1}-\psi_i)(\alpha-x)=0.$$
Therefore, if $\psi_i$ is symmetric, also $\psi_{i+1}$ is symmetric.
\end{proof}

\begin{theorem}
For $i\geq 0$, the function $\psi_i$ is a facet.
\end{theorem}
\begin{proof}
Since $\psi_i$ is a function that is piecewise linear, subadditive,
symmetric and has only two slopes, then, by Theorems~\ref{thm:gj}
and \ref{thm:2-slopes}, $\psi_i$ is a facet.
\end{proof}

\section{Analysis of the limit function}\label{SEC-limit}

Recall that $\psi$ is the function defined by
$$\psi(x)=\lim_{i\rightarrow \infty}\psi_i(x)$$
for every $x\in[0,1[$.\bigskip

Fact~\ref{fact:measures} implies the following.
\begin{fact}\label{fact:slope}
Let $\gamma = \alpha - \sum_{i=1}^{+\infty} 2^{i-1} \epsilon_i$.
Then $\gamma > 0$ by (\ref{series}) and $\gamma < \gamma_i$ for all
$i$, and the value $s_i$ of the positive slope in $\psi_i$ is
bounded above by $\frac{1-\gamma}{(1-\alpha)\gamma}$.
\end{fact}

We can now show the following lemma.

\begin{lemma}
For any $x$, the sequence $\{\psi_i(x)\}_{i=1,2,3,\ldots}$ is a
Cauchy sequence, and therefore it converges. Moreover, the sequence
of functions $\{\psi_i\}_{i=1,2,3,\ldots}$ converges uniformly to
$\psi$.
\end{lemma}

\begin{proof}  By Fact \ref{fact:measures}, there are
 $2^i$  intervals  where $\psi_i$ has  positive slope, each of
 length $\frac{\gamma_i}{2^{i}}$. Note that $|\psi_i(x) - \psi_{i+1}(x)|
\leq s_{i+1}\frac{\gamma_i}{2^{i}}$ since the values of the two
functions match at the ends of the positive-slope intervals of
$\psi_i$.

By Fact~\ref{fact:slope}, $s_{i+1} \leq
\frac{1-\gamma}{(1-\alpha)\gamma}$ and we know that $\gamma_i <
\alpha$. So $|\psi_i(x) - \psi_{i+1}(x)| \leq C\frac{1}{2^{i}}$
where $C = \alpha \frac{1-\gamma}{(1-\alpha)\gamma}$. Therefore,
$|\psi_n(x) - \psi_m(x)| \leq \sum_{i=n}^{m-1} C\frac{1}{2^{i}}$ if
$n<m$. We can bound this expression using

\[
\sum_{i=n}^{m-1} C\frac{1}{2^{i}} \leq \sum_{i=n}^{\infty}
C\frac{1}{2^{i}} = C\frac{1}{2^{n-1}}
\]

This implies that the sequence is Cauchy and hence convergent.
Moreover, since the bound on $|\psi_n(x) - \psi_m(x)|$ does not
depend on $x$, the above argument immediately implies that the
sequence of functions $\psi_i$ converges uniformly to $\psi$.
\end{proof}

This also implies the following corollary.

\begin{corollary}
The function $\psi$ is continuous.
\end{corollary}

\begin{proof}
$\psi_i$ is continuous for each $i \in \{1,2,3,\ldots\}$ by
construction. Since this sequence of functions converges uniformly
to $\psi$, $\psi$ is continuous~\cite{royden}.
\end{proof}

For each integer $i \geq 0$, define $S_i$ to be the subset of points of $]0,1[$
over which the function $\psi_i$ has negative slope. By Fact~\ref{fact:measures}, $S_i$ is the union of $2^i$ open intervals. Furthermore, by
construction $S_i\subseteq S_{i+1}$ for every $i\in\mathbb{N}$. The set $S\subseteq [0,1]$ defined by
$$S=\cup_{i=0}^\infty S_i,$$ is the set of points over which $\psi$ has negative slope, and it is an open set since it is the union of open intervals.

\begin{fact}\label{fact:dense}
The set $S$ is dense in $[0,1]$.
\end{fact}
\begin{proof} Let $a\in[0,1]$. We need to show that, for any
$\delta>0$, there exists $b\in S$ such that $|a-b|<\delta$. Choose
$i$ such that $\frac{1}{2^i}<\delta$. If $\psi_i$ has negative slope
in $a$, then $a\in S$ and we are done. Thus $a$ is in a positive
slope interval of $\psi_{i}$. By Fact~\ref{fact:measures}, such an
interval has length $\frac{\gamma_i}{2^i}$, hence there exists a
point $b$ in a negative slope interval of $\psi_{i}$, and thus in
$S$, such that $|a-b|\leq \frac{\gamma_i}{2^i}<\delta$, since
$\gamma_i < 1$.
\end{proof}

\begin{fact} The function $\psi$ is not piecewise linear.\end{fact}
\begin{proof}
Suppose by contradiction that $\psi$ is piecewise linear. Then, for some $\varepsilon>0$, the restriction of $\psi$ to $]0,\varepsilon[$ is linear. By Fact~\ref{fact:dense}, $]0,\varepsilon[$ contains a point of $S$, therefore $\psi$ has constant negative slope in $]0,\varepsilon[$. Since $\psi(0)=0$, then $\psi(x)<0$ for every $x\in]0,\varepsilon[$, a contradiction.
\end{proof}

\begin{lemma}\label{lemma:psi-subadditive}
The function $\psi$ is a minimal valid function.
\end{lemma}

\begin{proof} By Theorem~\ref{thm:gj}, we only need to show that $\psi$ is subadditive and symmetric.
This follows from pointwise convergence.
\end{proof}

We finally show that the function $\psi$ is a facet.

\begin{theorem}
The function $\psi$ is a facet for the problem (\ref{SI}).
\end{theorem}

\begin{proof}

We will use the Facet Theorem (Theorem~\ref{thm:facet_thm}). We show
that if a minimal valid function satisfies the set of equalities
$E(\psi)$, then it coincides with $\psi$ everywhere.

Consider any minimal valid function $\phi$ that satisfies the set of
equalities $E(\psi)$. Therefore, if $\psi(u) + \psi(v) = \psi(u+v)$,
then $\phi(u) + \phi(v) = \phi(u+v)$.

We first show the following.\medskip

\noindent{\em Claim 1. $\phi(x) = \psi(x) \textrm{ for all } x\in S$ }\medskip

For any maximal interval $]a,b[$ over which the graph of $\psi$ has
a negative slope, consider the following intervals : $U=[(a+b)/2,
b]$, $V=[1-((b-a)/2), 1]$ and therefore $U+V = [a,b]$. It is easy to
see that $\psi(u) + \psi(v) = \psi(u+v)$ for $u\in U$, $v\in V$.
This implies $\phi(u)  + \phi(v) = \phi(u+v)$. Now
Lemma~\ref{interval_lemma} (the Interval Lemma) implies that $\phi$
are straight lines over $U, V$ and $U+V$.

We now use an inductive argument to prove that not only do the
slopes of $\psi$ and $\phi$ coincide on intervals where the slope of
$\psi$ is negative, in fact $\psi(x) = \phi(x)$ for all $x$ in these
intervals.

Every maximal segment $s$ with negative slope in $\psi$ also appears
in $\psi_i$ for some $i$ (i.e. $s \subseteq S_i$ for some $i$). Let
$index(s)$ be the least such $i$. We prove that $\psi(x)=\phi(x)$
for every $s$ with negative slope by induction on $index(s)$.
$\phi(\alpha) = \psi(\alpha) = 1$ and $\phi(0) = \psi(0) = 0$ since
$\phi$ is assumed to be a valid inequality. This implies that $\phi$
is the same as $\psi$ in the range $[\alpha,1]$. This proves the
base case of the induction.

By the induction hypothesis, we assume the claim is true for
negative-slope segments $s$ with $index(s) = k$. Consider all
negative-slope segments $s$ with $index(s)=k+1$. Amongst these
consider the segment $s_c$ which is closest to the origin. Let the
midpoint of this segment be $m$. We know that $2m$ is the start of a
negative-slope segment $s'$ in $\psi$ with $index(s') = k$. By
construction, $\psi(m) + \psi(m) = \psi(2m)$. So $\phi(m) + \phi(m)
= \phi(2m)$. From the induction hypothesis, we know that $\psi(2m) =
\phi(2m)$ and so $\phi(m) = \frac{1}{2}\phi(2m) =
\frac{1}{2}\psi(2m) = \psi(m)$.

Now consider any other negative-slope segment $s$ with $index(s) =
k+1$ and let its midpoint be $m_s$. Note that $m_s + m$ is the start
of a negative-slope segment $s'$ with $index(s') = k$. So
\begin{equation}   \label{eq1}
\phi(m_s + m) = \psi(m_s + m)
\end{equation}
\noindent because of the inductive hypothesis. Note that $\psi(m_s +
m) = \psi(m_s) + \psi(m)$ by construction. So, $\phi(m_s + m) =
\phi(m_s) + \phi(m)$. Since we showed that $\phi(m) = \psi(m)$,
(\ref{eq1}) implies that $\phi(m_s) = \psi(m_s)$. Since the values
coincide at the midpoints of these segments and the slopes of the
segments are the same, $\phi(x)=\psi(x)$ for any $x$ in the domain
of these segments. This concludes the proof of Claim~1.\medskip

We now use Claim~1 to show that
\begin{equation}\label{eq:less}
\phi(x) \leq \psi(x) \textrm{ for all
} x\in [0,1[.
\end{equation}
 By Lemma \ref{lemma:psi-subadditive} $\psi$ is a minimal valid function and $\phi$ is assumed to be minimal,
 thus symmetry combined with (\ref{eq:less}) would imply that
$\phi(x) = \psi(x) \textrm{ for all } x\in [0,1[$.
\bigskip

\noindent{\em Claim 2.
Let $\bar x \in ]0,1[$. For every positive integer $n$, there
exists $y_n, z_n \in S$ such that $0 < z_n < \frac{1}{n}$ and $y_n +
z_n = \bar x$.
}
\medskip

Since $S$ is open and, by Fact~\ref{fact:dense}, it is dense in $[0,1]$,
for any positive integer $n$ there exists
an open interval $I \subseteq S$ such that $0 < \bar x - y <
\frac{1}{n}$ for all $y\in I$. Let $I = ]u_1, u_2[$ with $u_1 \neq
u_2$. Also, since $S$ is dense in $[0,1]$, there exists $z \in S \cap
]\bar x - u_2, \bar x - u_1[$. Let this $z$ be $z_n$ and $y_n = \bar
x - z \in I$. Since $I \subseteq S$, we have that $y_n \in S$. This concludes the proof of Claim~2.\medskip

Note that the sequence $\{y_n\}$ converges to $\bar x$ and $\{z_n\}$
converges to 0. For every positive integer $n$ we have
\[
\begin{array}{crcl}
& \phi(\bar x)  & = & \phi(y_n + z_n) \\
& & \leq & \phi(y_n) + \phi(z_n) \qquad (\textrm{By subadditivity of }\phi)\\
& & = &\psi(y_n) + \psi(z_n) \\
\Rightarrow & \phi(\bar x) & \leq & \lim_{n \rightarrow \infty}
(\psi(y_n) + \psi(z_n))\\
& & = & \psi(\bar x) \qquad (\textrm{By continuity of }\psi)
\end{array}
\]

\end{proof}
\section*{Acknowledgment}

We are particularly indebted to Santanu Dey for fruitful discussions about the Interval Lemma and its variants.
We also thank the referees for many helpful comments and suggestions.  \old{
\section{Conclusion}\label{SEC-conclusion}

The definition of valid function given by Gomory and
Johnson~\cite{gj}, requires a valid function to be continuous. If we
drop this assumption in the definition, then there are extreme
functions that are not continuous, as shown by Dey et
al.~\cite{DRLM}. The continuity assumption is crucial in our proof,
when showing that $\psi$ is a facet. Indeed, when showing that the
set of equalities $E(\psi)$ has only $\psi$ as a solution, we show
that for any other valid function $\phi$ satisfying $E(\psi)$ the
functions $\psi$ and $\phi$ coincide on a dense subset of the unit
interval, and thus they coincide everywhere by continuity of $\phi$
and $\psi$. }


\begin{thebibliography}{99}



\bibitem{acz} J. Acz\'el, {\it Lectures on functional equations and their applications}, Academic Press, New York, 1966.

\bibitem{DRLM} S.S. Dey, J.-P. P. Richard, Y. Li, L. A. Miller, On
the extreme inequalities of infinite group problems, {\it
Mathematical Programming} {\bf 121}, 145-170 (2010).

\bibitem{gj} R. E. Gomory, E. L. Johnson, Some Continuous Functions Related to Corner Polyhedra, Part I, {\it Mathematical Programming} {\bf 3}, 23-85 (1972).

\bibitem{gj2} R. E. Gomory, E. L. Johnson, Some Continuous Functions Related to Corner Polyhedra, Part II, {\it Mathematical Programming} {\bf 3}, 359-389 (1972).

\bibitem{sc} R. E. Gomory, E. L. Johnson, T-space and Cutting planes, {\it Mathematical Programming} {\bf 96}, 341-375 (2003).

\bibitem{kf} K. Kianfar, Y. Fathi, Generalized mixed integer rounding inequalities: facets for infinite group
polyhedra, {\it Mathematical Programming} {\bf 120}, 313-346 (2009).

\bibitem{royden} W. Rudin, {\it Principles of Mathematical Analysis}, 3rd ed., McGraw-Hill,
1976.

\end{thebibliography}
\end{document}